\documentclass[a4paper,12pt]{article}
\usepackage[margin=1in]{geometry}  

\usepackage{graphicx}              
\usepackage{amsmath}
\usepackage{amscd}               
\usepackage{amsfonts} 
\usepackage{amssymb}             
\usepackage{amsthm}                

\newtheorem{thm}{Theorem}[section]
\newtheorem{defn}[thm]{Definition}
\newtheorem{lem}[thm]{Lemma}
\newtheorem{prop}[thm]{Proposition}
\newtheorem{cor}[thm]{Corollary}

\newcommand{\RR}{\mathbb{R}}      
\newcommand{\ZZ}{\mathbb{Z}}      
\newcommand{\NN}{\mathbb{N}}

\newcommand{\EE}{\mathbb{E}}     

\newcommand{\lfl}{\left\lfloor }  
\newcommand{\rfl}{\right\rfloor} 

\begin{document}

\title{Inversion of signature for paths of bounded variation}
\author{Terry J. Lyons \and Weijun Xu}

\maketitle

\abstract{
We develop two methods to reconstruct a path of bounded variation from its signature. The first method gives a simple and explicit expression of any axis path in terms of its signature, but it does not apply directlty to more general ones. The second method, based on an approximation scheme, recovers any tree-reduced path from its signature as the limit of a uniformly convergent sequence of lattice paths. 
}

\bigskip

\section{Introduction}

Paths have important roles in many areas of mathematics. A path is a continuous function $\gamma$ that maps a non-empty interval $J \subset \RR$ into a metric space $(V,d_{V})$. A basic property of a path $\gamma$ on an interval $J$ is its length, $|\gamma|_{J}$, which can be defined as follows: 
\begin{align*}
|\gamma|_{J}:= \sup_{\mathcal{P} \subset J}{\sum_{i}d_{V}(\gamma(u_i), \gamma(u_{i+1}))}, 
\end{align*}
where the supremum is taken over all finite partitions $\mathcal{P}=\{s=u_{0}< u_{1}< \cdots <u_{n}=t\}$ of the interval $J=[s,t]$. It is clear that $|\gamma|_{J}$ is independent of the choice of parametrization. When there is no possible confusion, we will drop the subscript $J$, and use $|\gamma|$ to denote the length of the whole path. 

The mesh of a partition $\mathcal{P}$, $\left\| \mathcal{P} \right\|$, is defined by $\left\| \mathcal{P} \right\|:= \max_{i}|u_{i+1}-u_{i}|$. By triangle inequality and the continuity of $\gamma$, one can equivalently define $|\gamma|_{J}$ by
\begin{align} \label{limit}
|\gamma|_{J} = \lim_{\left\| \mathcal{P} \right\| \rightarrow 0}{\sum_{i}d_{V}(\gamma(u_i), \gamma(u_{i+1}))}, 
\end{align}
where the limit exists independent of the actual partition as long as the mesh tends to $0$. 

If $V$ is a Banach space, one can rewrite $d_{V}(\gamma(u_{i}), \gamma(u_{i+1}))$ as $|\gamma(u_{i+1})-\gamma(u_i)|$, where $|\cdot|$ is the Banach space norm. 

Paths $\gamma: J \rightarrow V$ of finite length are denoted as elements of $BV_{J}(V)$; they are also called paths with bounded variation. 

For any path $\alpha: [0,s] \rightarrow V$ and $\beta: [0,t] \rightarrow V$, we can form the concatenation $\alpha * \beta: [0,s+t] \rightarrow V$, as follows: 
\begin{equation*} 
\alpha*\beta(u) := \left \{
\begin{array}{rl}
&\alpha(u), u \in [0,s]\\
&\beta(u-s)+\alpha(s)-\alpha(0), u \in [s,s+t]
\end{array} \right., 
\end{equation*}
and similarly, the decomposition of one path into two can be carried out in the same fashion. 

For any path $\gamma: [s,t] \rightarrow V$, the path "$\gamma$ run backwards", $\gamma^{-1}$, is defined as: 
\begin{align*}
\gamma^{-1}(u) :=\gamma (s+t-u), u \in [s,t], 
\end{align*}
and the trajectories of $\gamma * \gamma^{-1}$ cancel out each other. 

Concatenation and "backwards" of paths of bounded variation are still paths of bounded variation. In fact, we have $|\alpha * \beta| = |\alpha| + |\beta|$, and $|\gamma^{-1}| = |\gamma|$. 

If $\gamma \in BV(\RR^{d})$, then one can define a differential equation driven by $\gamma$ as follows: 
\begin{align} \label{differential equation}
dy(t) = f(y(t)) \cdot d\gamma(t),\qquad y(0)=a, 
\end{align}
where $f$ is a vector field. If $f$ is Lipschitz, then equation \eqref{differential equation} has a unique solution $y(t)$. One may seek the properties of $\gamma$ that determines the value $y(t)$ given the initial value $y(0)=a$, and this leads to the concept of the signature of a path (\cite{Lyons 1998}, \cite{Hambly and Lyons}): 

\begin{defn}
Let $\gamma: [s,t] \rightarrow V$ be a path of bounded variation, where $V$ is also a vector space with a countable basis. The signature of $\gamma$, $X_{s,t}(\gamma)$, is defined as: 
\begin{align*}
X_{s,t}(\gamma) = 1 + X_{s,t}^{1}(\gamma) + \cdots + X_{s,t}^{n}(\gamma) + \cdots,  
\end{align*}
where
\begin{align} \label{tensor product}
X_{s,t}^{n}(\gamma) = \int_{s<u_{1} \cdots < u_{n}<t}{d\gamma(u_1) \otimes \cdots \otimes d\gamma(u_{n})}
\end{align}
as an element in $V^{\otimes n}$. 
\end{defn}

Let $(e_{1}, e_{2}, \cdots, )$ be a basis of $V$, then $\gamma$ can be written as $(\gamma_{1}, \gamma_{2}, \cdots)$. If $w = e_{i_1} \cdots e_{i_n}$ be a word of length $n$, we write
\begin{align*}
C_{s,t}(w)=C_{s,t}(e_{i_1} \cdots e_{i_n}) = \int_{s < u_{1} < \cdots < u_{n} <t}{d\gamma_{i_1}(u_1) \cdots d\gamma_{i_n}(u_n)}
\end{align*}
as the coefficient of $w$. As all words of length $n$ form a basis of $V^{\otimes n}$, we can rewrite $X_{s,t}^{n}(\gamma)$ as the linear combination of basis elements: 
\begin{align} \label{power series}
X_{s,t}^{n}(\gamma) = \sum_{|w|=n}C_{s,t}(w)w, 
\end{align}
where the sum is taken over all words of length $n$. 

Reparametrizing $\gamma$ does not change its signature. The signature contains important information about the path. For example, the first term, $X_{s,t}^{1}$, produces the increment of $\gamma$. 

The study of the signature of a path dates back to Chen \cite{Chen 1957}, where he associated a noncommutative power series to each piecewise smooth path in $\RR^{d}$, or on a Riemannian manifold $M$, with iterated integrals along the path as coefficients. This power series, in the form of \eqref{power series}, was obtained by solving \eqref{differential equation} using Picard's interation for the case when $f$ is linear. He showed that the logarithm (as a power series) of this power series is a Lie element, and obtained a generalized Campbell-Baker-Hausdorff formula. In his subsequent work \cite{Chen 1958}, he proved that two irreducible\protect\footnote{Loosely speaking, a path is irreducible if no part of it can be written as the form $\gamma * \gamma^{-1}$.} piecewise smooth paths with the same starting point have the same associated power series if and only if they differ by a parametrization. Lator on, he used this power series as a basic tool to compute the homology and cohomology of various path spaces on a manifold. A detailed disccusion on the role of such iterated integrals in relating the analysis on a manifold and the homology of its path spaces can be found in \cite{Chen 1977}. 

Half a century later, Lyons \cite{Lyons 1998} formally introduced the notion of signature, and defined this power series as the signature of the path in the form of \eqref{tensor product}. He also generalized the multiplicativity of signautre (first discovered by Chen \cite{Chen 1957}) to the notion of multiplicative funtionals, and showed that any multiplicative functional in $T^{(1)}$ with finite $1$-variation must be the signature of some path of finite length. 

Hambly and Lyons \cite{Hambly and Lyons} introduced the notion of tree-like paths, which are generalizations of paths of the form $\gamma_{1} * \gamma_{1}^{-1} * \cdots * \gamma_{n} * \gamma_{n}^{-1}$. Two paths $\gamma, \tau \in BV(\RR^{d})$ are equivalent if $\gamma * \tau^{-1}$ is tree-like. This equivalence relation defines the quotient group $BV(\RR^{d})/ \sim$. It is a group with concatenation as the multiplication. Within every equivalent class, there is a unique path of minimal length, called the reduced path (or irreducible path). They used various methods in analysis and hyperbolic geometry to get quantitative estimates for Chen's theorem extended to paths of bounded variation, and obtained as a corollary that two paths in $\RR^{d}$ of bounded variation have the same signature if and only if they are equivalent. 

After the work of Hambly and Lyons, a natural question is how to reconstruct the reduced path of bounded variation from its signature. This is known as the inversion problem. Hambly and Lyons \cite{Hambly and Lyons} gave a formula to recover the length of $C^{3}$ paths parametrized at unit speed by looking at the asymptotic behavior of $\left\| X^{n}(\gamma) \right\|$. This assumption can also be weakened to paths with continuous derivatives, but the conclusions will also be weaker. 

In this paper, we develop two methods to invert the signature. The first method gives an explicit inversion formula for axis paths, which are generalizations of integer lattice paths. But this method depends on the special structure of axis paths, and does not apply directly to more general ones. 

The second method is based on approximation of paths of bounded variation by lattice paths. Suppose $X = (1, X^{1}, \cdots, X^{n}, \cdots)$ is the signature of a tree-reducedpath $\gamma$ with length $L$. We will show that one can construct a sequence of lattice paths $\hat{\gamma}^{(N)}$, with step size $\frac{1}{2^N}$ and length at most $L$, such that for every $n$, $X^{n}(\hat{\gamma}^{(N)})$ converges to $X^{n}$. One can show further that, when properly located and parametrized, the sequence $\hat{\gamma}^{(N)}$ converges uniformly to $\gamma$. Thus, this approximation scheme asymptotically recovers any path of bounded variation from its signature by approaching it uniformly using a sequence of lattice paths. At the heart of our argument is an estimate of the difference of the signatures of two paths in terms of their lengths and uniform distance (theorem \ref{bounding the difference of signatures in terms of uniform distance}).

The main results in this paper are the following three theorems: 

\begin{flushleft}
\textbf{Theorem \ref{inversion for axis paths}}  \textit{Fix a finite axis path $\gamma$. Let $w = e_{i_1} \cdots e_{i_{k}} \cdots e_{i_n}$ be the unique longest square free word such that $C(w) \neq 0$. Let $w_{k} = e_{i_1} \cdots e_{i_k}^{2} \cdots e_{i_n}$, then we can write $\gamma$ as
\begin{align*}
\gamma = r_{1}e_{i_1} + \cdots + r_{n}e_{i_n}, 
\end{align*}
where $r_k = \frac{2C(w_k)}{C(w)}$, and the sum is noncommutative. }
\end{flushleft}

\begin{flushleft}
\textbf{Theorem \ref{uniform approximation by lattice paths}} \textit{Let $\gamma: [0,1] \rightarrow \RR^{d}$ be a path with $l_1$ length $L$. For any integer $N$, there exists a lattice path $\hat{\gamma}^{(N)}: [0,1] \rightarrow \RR^{d}$ with step size $\frac{1}{2^{N}}$ and length at most $L$ such that 
\begin{align*}
|\hat{\gamma}^{(N)}(t)-\gamma(t)| \leq \frac{d}{2^{N}} 
\end{align*}
for all $t \in [0,1]$.}
\end{flushleft}

\begin{flushleft}
\textbf{Theorem \ref{bounding the difference of signatures in terms of uniform distance}}  \textit{Let $\alpha, \beta: [0,1] \rightarrow \RR^{d}$ be two paths of finite length with a common control $\omega$\protect\footnote{Path $\gamma$ has finite length controlled by $\omega$ if $|\gamma|_{(s,t)} \leq \omega(s,t)$ for all $s \leq t$. The control $\omega$ is assumed to be jointly continuous, and additive in the sense that $\omega(s,t) = \omega(s,u) + \omega(u,t)$ for all $s \leq u \leq t$. }, and $|\alpha(t)-\beta(t)|< \epsilon$ for all $t \in [0,1]$. Then, 
\begin{align*}
\left\|  X_{s,t}^{n}(\alpha) - X_{s,t}^{n}(\beta) \right\| < 2 \epsilon \cdot \frac{(4 \omega(s,t))^{n-1}}{(n-1)!}
\end{align*}
for $s, t \in (0,1)$ and all $n \in \NN$.}
\end{flushleft}

The direct applications of this paper are to paths in $BV(\RR^{d})$, but many theorems still hold in a general Banach space. We focus on paths of bounded variation. For results and a more general theory of rough paths, we refer to the original work of Lyons \cite{Lyons 1998} and two comprehensive introductions (\cite{Lyons and Qian} and \cite{St. Flour}) for more details. 

We follow the notations in \cite{Hambly and Lyons} and \cite{Lyons 1998}. We assume the length of a path to be its $l_{1}$ length, unless otherwise specified. If $\gamma$ is the only path in our concern, we will use $X_{s,t}$ instead of $X_{s,t}(\gamma)$ to denote its signature. When we refer to the signature of the whole path, we will drop the subscripts $s,t$, and use $X$. 

The symbol $|\cdot|$ has several different meanings in different contexts. If $w$ is a word, then $|w|$ denotes the length of the word. If $E$ is a set, then $|E|$ denotes the cardinality of $E$. Finally, if $\gamma$ is a path in a Banach space with $l^{p}$ norm, then $|\gamma|_{p}$ denotes the length of the path in the Banach space.

\bigskip

\section{Preliminary results on the signature}

One essential property of the signature is its multiplicativity. It was first discovered by Chen \cite{Chen 1957}. We state it in the lemma below. 
\begin{lem}
Let $\gamma:[s,t] \rightarrow V$ be a path of bounded variation. Then, for any $r \in [s,t]$, $X_{s,r} \otimes X_{r,t}=X_{s,t}$. 
\end{lem}

The proof is an application of Fubini's theorem, after partitioning the domain of integration into $n$ disjoint parts. 

As mentioned in the introduction, the signature characterizes important information about the path. Hambly and Lyons \cite{Hambly and Lyons} recently proved that paths of bounded variations are completely characterized by their signatures up to tree-like equivalence. To state their theorem precisely, we first need a mathematical characterization of tree-like path, which was first introduced in \cite{Hambly and Lyons notes}, and then used in \cite{Hambly and Lyons}. 

\begin{defn}
$\gamma: [0,1] \rightarrow V$ is a tree-like path if there exists a continuous function $h: [0,1] \rightarrow \RR^{+}$ such that 
\begin{align*}
\left\| \gamma(t) - \gamma(s) \right\| \leq h(s) + h(t) - 2 \inf_{u \in [s,t]}h(u). 
\end{align*}
\end{defn}

The function $h$ is called the height function. Given a metric on the tree, one can think of $h(t)$ as the distance from $\gamma(0)$ to $\gamma(t)$ under that metric. More discussions on paths and trees can be found in \cite{Hambly and Lyons notes}. 

If $\gamma$ has bounded variation, then $h$ can be chosen to have bounded variation (theorem 12 of \cite{Hambly and Lyons}). With the aid of this characterization, they proved the following main theorem: 

\begin{thm}
Let $\alpha, \beta \in BV(\RR^{d})$. Then, $\alpha$ and $\beta$ have the same signature if and only if $\alpha * \beta^{-1}$ is tree-like. 
\end{thm}

We say two paths $\alpha$ and $\beta$ are equivalent if $\alpha * \beta^{-1}$ is tree-like, denoted by $\alpha \sim \beta$. It is clear that paths differing by a re-parametrization are equivalent. By the properties of tensors and multiplicativity of the signature, one can immediately conclude the following from the main theorem: 

\begin{cor}
$\sim$ defines an equivalence relation on $BV(\RR^{d})$. 
\end{cor}

This equivalence relation partitions $BV(\RR^{d})$ into equivalent classes. Within each class, there is a unique path with minimal length. 

\begin{cor}
For any $\gamma \in BV(\RR^{d})$, there exists a unique path $\tilde{\gamma}$ such that $X(\tilde{\gamma}) = X(\gamma)$, and if $\beta$ is any other path with the same signature, then $|\tilde{\gamma}| < |\beta|$. 
\end{cor}

The shortest path $\tilde{\gamma}$ corresponds to irreducible path in the finite case. We call it the tree-reduced path of $\gamma$. It is then natural to ask, how can one reconstruct the tree-reduced path of bounded variation from its signature. Hambly and Lyons showed that, one can recover the length of $\gamma$ by looking at the asymptotic behavior of $X_{s,t}^{n}$, provided that $\gamma$ is smooth enough ($C^{3}$): 

\begin{thm}
Let $\gamma: J \rightarrow \RR^{d}$ be a $C^{3}$ path of length $L$ parametrized at unit speed. Its signature is $X = (1, X^{1}, \cdots, X^{n}, \cdots)$. If $T^{n}$ is given the projective norm, then
\begin{align*}
\lim_{n \rightarrow \infty}{ \left\| \frac{n!   X^{n} }{L^{n}} \right\| } =1. 
\end{align*}
If $T^{n}$ is given the Hilbert Schmidt norm, then the limit 
\begin{align*}
\lim_{n \rightarrow \infty} {\left\|  \frac{n! X^{n}}{L^{n}}  \right\|^{2}} = \EE [ \exp(\int_{0}^{1}{|W_{s}^{0}|^{2} \left\langle \gamma'(s), \gamma'''(s) \right\rangle ds}) ] \leq 1, 
\end{align*}
exists, where $W_{s}^{0}$ is the Brownian bridge in time $[0,1]$ starting and finishing at zero, and $\left\langle \cdot, \cdot \right\rangle$ is the standard inner product on $\RR^{d}$. The limit is strictly less than 1 unless $\gamma$ is a straight line. 
\end{thm}

This is a very strong result obtained by making strong assumptions. One could weaken the $C^{3}$ condition and get a weaker result. This is the following theorem. It is also proved in \cite{Hambly and Lyons}. 

\begin{thm}
Let $\gamma$ be a path of length $L$ parametrized at unit speed. Suppose its derivative is continuous. Let $b_{n} = \left\| n! X^{n}(\gamma) \right\|$, where $\left\| \cdot \right\|$ is the projective tensor norm. Then, the Poisson averages $C_{\alpha}$ of the $b_{n}$'s defined by
\begin{align*}
C_{\alpha}:= e^{- \alpha} \sum_{n=0}^{+ \infty}\frac{\alpha^{n}}{n!} b_{n}
\end{align*}
satisfy
\begin{align*}
\lim_{\alpha \rightarrow \infty} \frac{1}{\alpha} \log C_{\alpha} = L-1. 
\end{align*}
\end{thm}

In the next two sections, we will develop methods to recover the tree reduced path itself from its signature.

\bigskip

\section{Inversion for axis paths}

In this section, we give an expression of any finite axis path in $\RR^{d}$ in terms of its signature. 

\begin{defn}
$\gamma: [s,t] \rightarrow \RR^{d}$ is a (finite) axis path if its movements are parallel to the Euclidean coordinate axes, has finitely many turns, and each straight line component has finite length. 
\end{defn}

Any axis path has a unique reduced axis path; integer lattice paths are special cases of axis paths. An $\RR^{d}$ axis path can move in $d$ different directions (up to the sign). At time $0$, it starts to move along a direction $e_{i_1}$ for some distance $r_1$; then it turns a right angle, and moves along $e_{i_2}$ for a distance $r_2$, and so on, and stops after finitely many turns. Thus, an (unparametrized) axis path $\gamma$ can be represented as: 
\begin{align} \label{representation}
\gamma = r_{1}e_{i_1} + \cdots + r_{n}e_{i_n}
\end{align}
where $r_i$'s are real numbers, with the sign denoting the direction\protect\footnote{We mean $-r e_{j} = r e_{j}^{-1}$.}. The sum is noncommutative. If $\gamma$ is already in its reduced form, then it is clear that $i_{k} \neq i_{k+1}$, and we call $(e_{i_1}, \cdots, e_{i_n})$ the shape of $\gamma$. We introduce the notion of square free words to characterize the shape of an axis path. 

\begin{defn}
Let $w = e_{i_1} \cdots e_{i_n}$ be a word. We call it a square free word if $\forall k \leq n-1$, $i_{k} \neq i_{k+1}$. 
\end{defn}

If a path has shape $(e_{i_1}, \cdots, e_{i_n})$, then by multiplicativity, its signature can be written as
\begin{align} \label{signature exponential}
X(\gamma) = e^{r_{1}e_{i_1}} \cdots e^{r_{n}e_{i_n}}, 
\end{align}
where the product is noncommutative. Let $w=e_{i_1} \cdots e_{i_n}$, then we have: 
\begin{align} \label{nonzero square free}
C(w)=r_{1} \cdots r_{k} \cdots r_{n}, 
\end{align}
and in particular, $C(w) \neq 0$. Moreover, if $w'$ is any other square free word with $C(w') \neq 0$, then $w'$ has length at most $n-1$, and we can thus recover the shape of the path by looking at its square free words. We state this observation in the proposition below. 

\begin{prop}
For any finite axis path $\gamma$, there exists a unique square free word $w$ such that $C(w) \neq 0$, and if $w'$ is any other square free word with $C(w') \neq 0$, then $|w'| < |w|$. If $w = e_{i_1} \cdots e_{i_n}$, then $\gamma$ has shape $(e_{i_1}, \cdots, e_{i_n})$. 
\end{prop}

As a second step, we need to find out the $r_i$'s, assuming that we know the path has shape $(e_{i_1}, \cdots, e_{i_n})$. Let $w=e_{i_1} \cdots e_{i_k} \cdots e_{i_n}$, and $w_{k}=e_{i_1} \cdots e_{i_{k}}^{2} \cdots e_{i_n}$, then, 
\begin{align} \label{one square}
C(w_k) = \frac{1}{2}r_{1} \cdots r_{k}^{2} \cdots r_{n}, 
\end{align}
and compare with \eqref{nonzero square free}, we immediately get: 
\begin{align*}
r_{k} = \frac{2C(w_{k})}{C(w)}, 
\end{align*}
thus recovering the length of each straight line component. 

Combining the arguments above, we have the following theorem to invert a finite axis path from its signature: 

\begin{thm} \label{inversion for axis paths}
Let $\gamma$ be a finite axis path. Let $w = e_{i_1} \cdots e_{i_{k}} \cdots e_{i_n}$ be the unique longest square free word such that $C(w) \neq 0$. Let $w_{k} = e_{i_1} \cdots e_{i_k}^{2} \cdots e_{i_n}$, then we can write $\gamma$ as
\begin{align*}
\gamma = r_{1}e_{i_1} + \cdots + r_{n}e_{i_n}, 
\end{align*}
where $r_k = \frac{2C(w_k)}{C(w)}$, and the sum is noncommutative. 
\end{thm}

Thus, if an axis path has $n$ turns, at most $n+2$ terms in the signature are needed for inversion. For a lattice path with length $L$, it can have at most $L-1$ turns, so we only need the first $L+1$ terms in the signature to recover it. 

In practice, lattice paths are often generated by drawing uniformly randomly $n$ letters and their inverses from an alphabet, and putting them in a row in the order they are drawn. It is then intersting to ask about the number of turns in its reduced path. We give uniform counting measure on all lattice paths with $n$ steps, and let $T_n$ denote the number of turns in the reduced path. The asymptotics of $T_{n}$ for large $n$ was computed by Jiang and Xu in \cite{Jiang and Xu}.

\bigskip

\section{Inversion for paths of bounded variation}

Given the signature $X=(1, X^{1}, \cdots)$ of a tree-reduced path of bounded variation, we will find a sequence of lattice paths $\{\hat{\gamma}^{(N)}\}$ whose signatures converge to $X$. We show that $\{\hat{\gamma}^{(N)}\}$ also necessarily converges uniformly, and thus obtain the original path as the uniform limit of this sequence of lattice paths. In this approach, the first and most important problem is that it is not at all clear whether there exists such a sequence of lattice paths. But once the existence is verified, one can find this sequence by checking all the possibilities, as there are only finitely many candidates for each $\hat{\gamma}^{(N)}$. Therefore, we will devote most of our efforts in proving the existence of such a sequence. Our precedure is as follows. 

\begin{enumerate}
\item Let $\gamma$ be a path of length $L$. For every integer $N$, we construct a lattice path $\hat{\gamma}^{(N)}$ on the same time interval with $\gamma$ with step size $\frac{1}{2^{N}}$ and length at most $L$ such that
\begin{align*}
\left\|  \hat{\gamma}^{(N)}(t) - \gamma(t)  \right\| \leq \frac{d}{2^{N}}
\end{align*}
for all $t$. In particular, $\{\hat{\gamma}^{(N)}\}$ converges to $\gamma$ uniformly. 

\item We will show that each $\hat{\gamma}^{(N)}$ constructed above satisfies
\begin{align} \label{convergence of the signatures}
\left\| X^{(n)}(\hat{\gamma}^{(N)}) - X^{n}  \right\| < \frac{d}{2^{N-1}} \cdot \frac{(4L)^{n-1}}{(n-1)!}
\end{align}
for all $n$. 

\item Suppose $X$ is the signature of a tree-reduced path $\gamma$ with length $L$. Then, in light of the previous two steps, one can find a sequence of lattice paths satisfying \eqref{convergence of the signatures}. This sequence, if located and parametrized properly, converges uniformly to $\gamma$. Thus, we are able to asymptotically recover any tree-reduced path from its signature by approximating it uniformly using lattice paths. 

\end{enumerate}

Note that the first two steps provides a verification for the existence of the sequence of lattice paths mentioned at the beginning of this section. Step 3 is a simple consequence of the previous two. We now start with the first step.

\bigskip

\subsection{Lattice path approximation to paths of bounded variation}

Let $\gamma: [0,1] \rightarrow \RR^{d}$ be a one dimensional path with length $L$. For any $x \in \RR$, let $c(x)$ denote the cardinality of the level set $\{t| \gamma_{i}(t)=x\}$. $c(x)$ is also the $0$-dimensional Hausdorff measure of the level set. Then, by the coarea formula (theorem 1 of section 3.4 in \cite{Evans and Gariepy}), we have
\begin{align*}
\int_{- \infty}^{+ \infty}c(x)dx = L, 
\end{align*}
or equivalently, 
\begin{align*}
\sum_{k = - \infty}^{+ \infty}\int_{0}^{\frac{1}{2^N}}{c(x + \frac{k}{2^N})dx} = L
\end{align*}
for all $N$. Since $c$ takes values on positive integers, monotone convergence theorem implies
\begin{align*}
\int_{0}^{\frac{1}{2^N}}{\sum_{k = -\infty}^{+ \infty}c(x + \frac{k}{2^N})dx} = L
\end{align*}
for all $N$. We then have the following lemma: 

\begin{lem} \label{existence of the point}
Let $\gamma$ be a path in $\RR$ with length $L$. Then, for any $N$, there exists an $x \in \RR$ such that
\begin{align} \label{one dimensional point}
\sum_{k = - \infty}^{+ \infty}c(x + \frac{k}{2^{N}}) \leq L \cdot 2^{N}. 
\end{align}
\end{lem}

Now we let $\gamma$ live in $\RR^{d}$, and write $\gamma(t) = (\gamma_{1}(t), \cdots, \gamma_{d}(t))$. Since \eqref{one dimensional point} is true for all one dimensional paths, one can choose for each $i$ an $x_{i} \in \RR$ such that
\begin{align}
\sum_{k = - \infty}^{+ \infty}c(x_{i} + \frac{k}{2^{N}}) \leq l_{i} \cdot 2^{N}
\end{align}
for all $i$, where $l_{i} = |\gamma_{i}|$. This implies that, for each $i$, the number of times $\gamma$ hits the parallel hyperplanes $\tilde{x_i}=\{x_{i}+\frac{k}{2^{N}}\}_{k = -\infty}^{+\infty}$  is at most $l_{i} \cdot 2^{N}$. Let $x = (x_{1}, \cdots, x_{d})$, then the set of such $x$'s has positive Lebesgue measure, and one can choose an $x$ such that there is a unique $k = (k_1, \cdots, k_d) \in \ZZ^{d}$ satisfying
\begin{align*}
x + \frac{k_i}{2^{N}} < \gamma_{i}(0) < x + \frac{k_{i}+1}{2^N}, 
\end{align*} 
for all $i = 1, \cdots, d$. That is, the starting point of $\gamma$ is in the interior of a hypercube formed by the parallel hyperplanes $\tilde{x_i}=\{x_{i}+\frac{k}{2^{N}}\}_{k = -\infty}^{+\infty}$. 

We construct the lattice path $\hat{\gamma}^{(N)}$ as follows. Start at the center of the cube where $\gamma(0)$ lives; when $\gamma$ crosses a surface and moves to another (neighboring) cube, move $\hat{\gamma}^{(N)}$ one step of size $\frac{1}{2^N}$, from the current position to the center of that neighboring cube. When $\gamma$ does not move out of a cube, keep $\hat{\gamma}^{(N)}$ stayed at the center of a cube. Since $\gamma$ crosses the surfaces at most $L \cdot 2^{N}$ times, $\hat{\gamma}^{(N)}$ has at most $L \cdot 2^{N}$ steps. 

From that construction, one can parametrize $\hat{\gamma}^{N}$ such that $\hat{\gamma}^{N}(t)$ and $\gamma(t)$ are in the same hypercube of side length $\frac{1}{2^N}$ for all $t$. This leads to the following theorem. 

\begin{thm} \label{uniform approximation by lattice paths}
Let $\gamma: [0,1] \rightarrow \RR^{d}$ be a path with $l_1$ length $L$. For any integer $N$, there exists a lattice path $\hat{\gamma}^{(N)}: [0,1] \rightarrow \RR^{d}$ with step size $\frac{1}{2^{N}}$ and length at most $L$ such that 
\begin{align*}
|\hat{\gamma}^{(N)}(t)-\gamma(t)| \leq \frac{d}{2^{N}} 
\end{align*}
for all $t \in [0,1]$. 
\end{thm}

\bigskip

\subsection{Convergence of the signatures}

In this subsection, we will show that the signatures of the lattice paths constructed in the previous section converge to that of the original path. One crucial ingredient of the proof is an estimate of the difference of the signatures of two paths in terms of their lengths and uniform distance (theorem \ref{bounding the difference of signatures in terms of uniform distance}). 

Let $\alpha, \beta : [0,1] \rightarrow \RR^{d}$ be two paths of bounded variation with a common control $\omega$, and $|\alpha(t)-\beta(t)| < \epsilon$ for all $t \in [0,1]$. For any $N$, there exists a partition $\{0=t_{0}<t_{1}< \cdots < t_{N=1}\}$ of $[0,1]$ such that $\omega(t_{i},t_{i+1}) = \delta \leq \frac{L}{N}$ for all $i$, and one can decompose the paths as
\begin{align*}
\alpha = \alpha_{1} * \cdots * \alpha_{N}, \qquad \beta = \beta_{1} * \cdots * \beta_{N}, 
\end{align*}
where $\alpha_{i}$ and $\beta_{i}$ are restrictions of $\alpha$ and $\beta$ to time interval $[t_{i-1}, t_{i}]$.

By multiplicativity of the signature, we have
\begin{align*}
X(\alpha) &= X(\theta_{1} * \beta_{1} * \cdots * \theta_{N} * \beta_{N}), 
\end{align*}
where $\theta_{i} = \alpha_{i} * \beta_{i}^{-1}$. Then, $ \theta = \theta_{1} * \cdots * \theta_{N}$ satisfies
\begin{align} \label{uniform closeness}
\sup_{t}|\theta_{k+1} * \cdots * \theta_{k+p}(t)| < 2 \epsilon
\end{align}
for all $k$ and $l$ with $k+l \leq N$. Indeed, we have
\begin{align*}
X^{n}(\alpha) = \sum_{\sum_{k}{(i_{k}+ j_{k})} = n} X^{i_1}(\theta_{1}) \otimes X^{j_1} \cdots \otimes X^{i_{N}}(\theta_{N}) \otimes X^{j_{N}}(\theta_{N})
\end{align*}
Similar as before, we let
\begin{align*}
A_{N,n}^{(1)} = \{ (i_{k}, j_{k})_{k=1}^{N} | i_{k},j_{k} = 0 \phantom{1} \text{or} \phantom{1} 1 \phantom{1} \text{for all} \phantom{1} k, \phantom{1} \text{and} \phantom{1} i_{k} = 1 \phantom{1} \text{for some} \phantom{1} k  \}, \\
A_{N,n}^{(2)} = \{ (i_{k}, j_{k})_{k=1}^{N} | i_{k} \phantom{1} \text{or} \phantom{1} j_{k} \geq 2 \phantom{1} \text{for some} \phantom{1} k, \phantom{1} \text{and} \phantom{1} i_{k} \geq 1 \phantom{1} \text{ for some} \phantom{1} k  \}, 
\end{align*}
and 
\begin{align*}
S_{N,n}^{(1)} = \sum_{A_{N,n}^{(1)}} X^{i_{1}}(\theta_{1}) \otimes X^{j_{1}}(\beta_{1}) \otimes \cdots \otimes X^{i_{N}}(\theta_{N}) \otimes X^{j_{N}}(\beta_{N}), \\
S_{N,n}^{(2)} = \sum_{A_{N,n}^{(2)}} X^{i_{1}}(\theta_{1}) \otimes X^{j_{1}}(\beta_{1}) \otimes \cdots \otimes X^{i_{N}}(\theta_{N}) \otimes X^{j_{N}}(\beta_{N}), 
\end{align*}
then
\begin{align*}
X^{n}(\alpha) = X^{n}(\beta) + S_{N,n}^{(1)} + S_{N,n}^{(2)}. 
\end{align*}
Thus, we arrive in a similar situation as before. Before stating the main theorem, we first prove a lemma. 

\begin{lem}
Let $S_{N,n}^{(2)}$ be defined as above, and let $\delta < \frac{1}{2}$. Then, we have
\begin{align*}
\left\|  S_{N,n}^{(2)} \right\| \leq 2 \delta \frac{(4L + 2\delta(n-1))^{n-1}}{(n-2)!}
\end{align*}
\end{lem}
\begin{proof}
By definition, a typical term in the sum of $S_{N,n}^{(2)}$ is bounded by
\begin{align*}
\left\| X^{i_{1}}(\theta_{1}) \otimes X^{j_{1}}(\beta_{1}) \otimes \cdots \otimes X^{i_{N}}(\theta_{N}) \otimes X^{j_{N}}(\beta_{N}) \right\|  &\leq \frac{(2 \delta)^{i_{1}}}{i_{1}!} \cdot \frac{\delta^{j_1}}{j_{1}!} \cdots \frac{(2 \delta)^{i_N}}{i_{N}!} \cdot \frac{\delta^{j_N}}{j_{N}!} \\
&\leq (2 \delta)^{n}, 
\end{align*}
where we have used $\sum_{k}(i_{k} + j_{k}) = n$ in the last line. Now we count the elements in $A_{N,n}^{(2)}$. 
Let $A_{N,n} = A_{N,n}^{(1)} \cup A_{N,n}^{(2)}$, then it is clear that
\begin{align*}
|A_{N,n}| &= \begin{pmatrix} n+2N-1\\ 2N-1 \end{pmatrix} - \begin{pmatrix} n+N-1\\ N-1 \end{pmatrix} \\
\end{align*}
On the other hand, we have
\begin{align*}
|A_{N,n}^{(1)}| = \begin{pmatrix} 2N\\ n \end{pmatrix} - \begin{pmatrix} N\\ n \end{pmatrix}
\end{align*}
Combining them together, we get
\begin{align*}
|A_{N,n}^{(2)}| &= |A_{N,n}| - |A_{N,n}^{(1)}| \\
&\leq \begin{pmatrix} 2N+n-1\\ n \end{pmatrix} - \begin{pmatrix} 2N\\ n \end{pmatrix} \\
&\leq \frac{1}{n!} [(2N+n-1)^{n} - (2N)^{n}] \\
&\leq \frac{1}{n!} (n-1) n (2N+n-1)^{n-1} \\
&= \frac{(2N+n-1)^{n-1}}{(n-2)!}. 
\end{align*}
Thus, we have
\begin{align*}
\left\|  S_{N,n}^{(2)} \right\| \leq 2 \delta \frac{(4L + 2\delta(n-1))^{n-1}}{(n-2)!}
\end{align*}

\end{proof}

Now we are ready to prove our main estimate. 

\begin{thm} \label{bounding the difference of signatures in terms of uniform distance}
Let $\alpha, \beta: [0,1] \rightarrow \RR^{d}$ be two paths of finite $l^{1}$ length with a common control $\omega$, and $|\alpha(t)-\beta(t)| < \epsilon$ for all $t \in [0,1]$.  Then, we have
\begin{align} \label{inequality of uniform bound}
\left\|  X_{s,t}^{n}(\alpha) - X_{s,t}^{n}(\beta) \right\| < 3 \epsilon \cdot \frac{(3 \omega(s,t))^{n-1}}{(n-1)!}
\end{align}
for $s, t \in (0,1)$ and all $n \in \NN$. 
\end{thm}
\begin{proof}
In light of the previous lemma, by sending $\delta$ to $0$, it suffices to prove the bound for $\left\|  S_{N,n}^{(1)} \right\|$. A typical term in the sum of $S_{N,n}^{(1)}$, denoted by $W^{n}$, is the tensor product of $n$ ordered terms of $X^{1}(\theta_{i})$ and $X^{1}(\beta_{j})$'s, and consists of at least one $X^{1}(\theta_{i})$. Now fix $W^{n}$. Suppose the rightmost $X^{1}(\theta_{i})$ in the product appears at position $k$, then $W^{n}$ has the form
\begin{align*}
W^{n} = W^{n}(Y^{k-1},Z^{n-k}) = Y^{k-1}(\theta, \beta) \otimes X^{1}(\theta_{i}) \otimes Z^{n-k}(\beta), 
\end{align*}
where $Y^{k-1}(\theta, \beta)$ is the tensor product of the first $k-1$ ordered terms, each term being $X^{1}(\theta_{i})$ or $X^{1}(\beta_{j})$, and $Z^{n-k}(\beta)$ is the tensor product of the last $n-k$ terms, which by assumption are all of the form $X^{1}(\beta_{j})$. 

We first restrict to all $W^{n}$'s such that the rightmost $X^{1}(\theta_{i})$ appears at position $k$, where $k$ is any integer between $1$ and $n$. Denote this set as $E_{k}$. The collections of $W^{n}$'s with the same feasible choice $(Y,Z) = (Y^{k-1},Z^{n-k})$ os a sibset pf $E_{k}$, denoted by $E_{k}(Y,Z)$. For every feasible choice of $(Y^{k-1},Z^{n-k})$, we can sum up all $W^{n}(Y,Z)$'s in $E_{k}$ with the same realization $(Y^{k-1},Z^{n-k})$. By linearality, we can sum together the terms $X^{1}(\theta_{i})$ on the position $k$ to get
\begin{align*}
\sum_{E_{k}(Y,Z)} W^{n}(Y,Z) = Y^{n-1}(\theta, \beta) \otimes \bigg(\sum_{i=l+1}^{l+p} X^{1}(\theta_{i}) \bigg) \otimes Z^{n-k}(\beta), 
\end{align*}
where the sum is taken over $E_{k}(Y,Z)$. Now suppose the product $Y^{k-1}(\theta, \beta)$ consists of $p$ terms of $X^{1}(\theta_i)$'s and $k-1-p$ terms of $X^{1}(\beta_j)$'s. Since 
$\left\| X^{1}(\theta_{i}) \right\| \leq 2 \delta$, $\left\| X^{1}(\beta_{j}) \right\| < \delta$, and also by assumption
\begin{align*}
\left\|  \sum_{i=l+1}^{l+p} X^{1}(\theta_{i})  \right\| < 2 \epsilon, 
\end{align*}
it is straightforward that
\begin{align*}
\left\| W^{n}(Y^{k-1},Z^{n-k})  \right\| < 2 \epsilon \cdot 2^{p} \cdot \delta^{n-1}. 
\end{align*}

We first consider the case $k=n$. In this case, if there are $p$ terms of $X^{1}(\theta_i)$'s and $n-1-p$ terms of  $X^{1}(\beta_{j})$'s in the product of the first $n-1$ terms, the number of such chocies are
\begin{align*}
\begin{pmatrix} N-1\\ p \end{pmatrix} \cdot \begin{pmatrix} N-1\\ n-1-p \end{pmatrix} < \frac{N^{n-1}}{p!(n-1-p)!}. 
\end{align*}
Thus, we have
\begin{align*}
\left\| \sum_{E_{n}}W^{n}(Y^{n-1},Z^{0})  \right\| &\leq \sum_{E_{n}} \left\| W^{n}(Y^{n-1},Z^{0}) \right\| \\
&\leq \sum_{p=0}^{n-1} 2 \epsilon \cdot 2^{p} \cdot \delta^{n-1} \cdot \frac{N^{n-1}}{p!(n-1-p)!} \\
&= 2 \epsilon \cdot \frac{(3L)^{n-1}}{(n-1)!}
\end{align*}

Now we consider the case $k \leq n-1$. Same as before, we assume there are $p$ terms of $X^{1}(\theta_{i})$'s and $k-1-p$ terms of $X^{1}(\beta_{j})'s$ in the first $k-1$ terms. In addition, we let call the $k+1$-th term $X^{1}(\beta_{l})$. The above setting implicit assumes the following relationships: 
\begin{align*}
\frac{k-1}{2} \leq l \leq N-(n-k), \qquad \max\{0, k-1-l\} \leq p \leq \min\{k-1,l\}. 
\end{align*}
With this setting, we have
\begin{align*}
\left\|  \sum_{k=1}^{n-1} \sum_{E_{k}} W^{n}(Y^{k-1},Z^{n-k})  \right\|  &\leq  2 \epsilon \cdot \delta^{n-1} \sum_{k=1}^{n-1} \sum_{l = \lfl \frac{k-1}{2} \rfl}^{N-(n-k)} \sum_{p = \max\{0,k-1-l\}}^{\min\{k-1,l\}} 2^{p} \cdot \begin{pmatrix} l\\ p \end{pmatrix} \cdot \begin{pmatrix} l\\ k-1-p \end{pmatrix} \cdot \begin{pmatrix} N-l\\ n-k-1 \end{pmatrix} \\
&\leq 2 \epsilon \cdot \delta^{n-1} \sum_{k=1}^{n-1} \sum_{l=0}^{N-1} \sum_{p=0}^{k-1} 2^{p} \cdot \frac{l^{k-1}}{p! (k-1-p)!} \cdot \frac{(N-l)^{n-k-1}}{(n-k-1)!} \\
&= 2 \epsilon \cdot \delta^{n-1} \sum_{l=0}^{N-1} \sum_{k=1}^{n-1} \frac{(3l)^{k-1}}{(k-1)!} \cdot \frac{(N-l)^{n-2-(k-1)}}{(n-2-(k-1))!} \\
&= 2 \epsilon \cdot \frac{\delta^{n-1}}{(n-2)!} \sum_{l=0}^{N-1}(N+2l)^{n-2} \\
&\leq 2 \epsilon \cdot \frac{\delta^{n-1}}{(n-2)!} \int_{1}^{N}{(N+2x)^{n-2}dx} \\
&\leq \epsilon \cdot \frac{(3L)^{n-1}}{(n-1)!}
\end{align*}

Combining the above estimate with the case $k=n$, we get
\begin{align*}
\left\| S_{N,n}^{(1)} \right\| \leq 3 \epsilon \cdot \frac{(3L)^{n-1}}{(n-1)!}. 
\end{align*}

Sending $N \rightarrow +\infty$, by the previous lemma, we conclude that
\begin{align*}
\left\| X^{\alpha} - X^{n}(\beta) \right\| \leq 3 \epsilon \cdot \frac{(3L)^{n-1}}{(n-1)!}. 
\end{align*}

\end{proof}

The right hand side of \eqref{inequality of uniform bound} appears like the bound for signature of degree $n$. This is because one needs to compensate the uniform distance $\epsilon$ with one degree of signature. It is worthy to note that, the scale of $n-1$ in the bound cannot be both improved to $n$. To see this, consider the one dimensional path $\gamma(t) = \epsilon t, t \in [0,1]$. Then, 
\begin{align*}
\left\| X_{s,t}^{n}(\gamma) \right\| = \epsilon^{n} \frac{(t-s)^{n}}{n!}. 
\end{align*}
If the bound were at the scale of $n$, then it would be $\epsilon^{n+1} \frac{(C(t-s))^{n}}{n!}$ for some constant $C$ indepedent of $n$. One can choose $\epsilon$ small enough so that $\left\|  X_{s,t}^{n}(\gamma) \right\|$ exceeds this bound. So, it is not possible to improve the right hand side of \eqref{inequality of uniform bound} to the scale of $n$.

For the lattice path approximation in the previous section, we have $|\hat{\gamma}^{N}(t)-\gamma(t)| \leq \frac{1}{2^N}$, so we immediately get the following corollary.

\begin{cor}
Let $\gamma$ be a path in $\RR^{d}$ of length $L$, and $\{\hat{\gamma}^{(N)}\}$ be a sequence of lattice paths, each with step size $\frac{1}{2^{N}}$ and length at most $L$, and satisfying
\begin{align*}
|\hat{\gamma}^{(N)}(t) - \gamma(t)| < \frac{d}{2^{N}}
\end{align*}
for all $t$, as constructed in the previous section. Then, 
\begin{align*}
\left\|  X^{n}(\hat{\gamma}^{(N)}) - X^{n}(\gamma) \right\| < \frac{3d \cdot (3L)^{n-1}}{2^{N} \cdot (n-1)!}. 
\end{align*}
\end{cor}

\bigskip

\subsection{Inversion}

Let $X = (1, X^{1}, X^{2}, \cdots)$ be the signature of a path $\gamma$. Suppose further that $\gamma$ is tree-reduced and has $L$\protect\footnote{Note that in the approximation scheme, we require the lengths of all lattice paths to be bounded by a real number $L'$. If $L' > L$, we can lower the bound $L'$ until it reaches $L$. So in fact we can approach $L$ arbitrarily closely. }. By the construction and estimates previous two sections, for every $N$, there exists a lattice path $\hat{\gamma}^{(N)}$ with step size $\frac{1}{2^{N}}$ and length at most $L$ such that \eqref{convergence of the signatures} holds. The uniform bound on the lengths implies that there are only finitely many candidates for each $\hat{\gamma}^{(N)}$, so it guarantees one can find such a sequence by checking all the possibilities at each level $N$. 

This sequence $\{\hat{\gamma}^{(N)}\}$, if all starting from the origin, and parametrized at unit speed, is uniformly bounded and equicontinuous. Thus, it contains a uniformly convergent subsequence. Denote its limit by $\hat{\gamma}$. By theorem \ref{bounding the difference of signatures in terms of uniform distance}, for each $n$, we have
\begin{align*}
X^{n}(\hat{\gamma}^{(N)}) \rightarrow X^{n}(\hat{\gamma})
\end{align*}
as $N \rightarrow +\infty$, and thus we get
\begin{align*}
X(\hat{\gamma}) = X(\gamma), 
\end{align*}
so $\hat{\gamma} \sim \gamma$ by the uniqueness theorem of Hambly and Lyons. 

On the other hand, as a uniform limit, we have
\begin{align*}
|\hat{\gamma}| \leq L, 
\end{align*}
but $L$ is the length of the tree reduced path $\gamma$, so we must have
\begin{align*}
\hat{\gamma} = \gamma. 
\end{align*}
Thus, every uniformly convergent subsequence of $\{\hat{\gamma}^{(N)}\}$ has the same limit $\gamma$, and so the whole sequence must also converge uniformly to $\gamma$. In this way, we can approach $\gamma$ arbitrarily close in uniform distance, and invert the signature asysmptotically.

\bigskip

\textsc{Mathematical Institute, University of Oxford, 24-29 St. Giles', Oxford, OX1 3LB, UK}

\bigskip

\textit{Email address}: tlyons@maths.ox.ac.uk

\bigskip

\textit{Email address}: xu@maths.ox.ac.uk


\bigskip





\begin{thebibliography}{99}
\bibitem{Chen 1957} K-T Chen, Integration of paths, geometric invariants and a generalized Baker-Hausdorff formula, \textit{Annals of Mathematics}, Vol.65, No.1 (1957), pp.163-178. 
\bibitem{Chen 1958} K-T Chen, Integration of paths - a faithful representation of paths by noncommutative formal power series, \textit{Transactions of the A.M.S.}, Vol.89, No.2 (1958), pp.395-407. 
\bibitem{Chen 1977} K-T Chen, Iterated path integrals, \textit{Bulletin of the A.M.S}, Vol.83, No.5 (1977), pp.831-879. 
\bibitem{Evans and Gariepy} L.C.Evans, R.F.Gariepy, Measure Theory and Fine Properties of Functions, CRC Press, 1991. 
\bibitem{Hambly and Lyons notes} B.M.Hambly, T.J.Lyons, Some notes on trees and paths, unpublished manuscript, available at http://arxiv.org/abs/0809.1365, 2008. 
\bibitem{Hambly and Lyons} B.M.Hambly, T.J.Lyons, Uniqueness for the signature of a path of bounded variation and the reduced path group, \textit{Annals of Mathematics}, Vol.171, No.1 (2010), pp.109-167. 
\bibitem{Jiang and Xu} Y.Jiang, W.Xu, On number of turns in reduced random lattice paths, preprint, available at http://arxiv.org/abs/1007.3507, 2010. 
\bibitem{Lyons 1998} T.J.Lyons, Differential equations driven by rough signals, \textit{Rev. Mat. Iberoamericana}, Vol.14, No.2 (1998), pp.215-310. 
\bibitem{Lyons and Qian} T.J.Lyons, Z.Qian, System Control and Rough Paths, \textit{Oxford Mathematical Monographs}, Oxford University Press, 2002. 
\bibitem{St. Flour} T.J.Lyons, M.J.Caruana, T.L\'{e}vy, Differential Equations Driven by Rough Paths, \textit{Lecture Notes in Mathematics} 1908, Springer Berlin, 2006. 
\bibitem{Stanley} R.P.Stanley, Enumerative Combinatorics, Vol.1, second edition, \textit{Cambridge Studies in Advanced Mathematics}, Cambridge University Press, 2000. 


\end{thebibliography}
\end{document}